\documentclass[oneside, 12pt]{amsart}
\usepackage{amsmath}
\usepackage{amssymb}
\usepackage{color}
\usepackage[usenames,dvipsnames,x11names,svgnames]{xcolor}
\usepackage[colorlinks=true,linkcolor=NavyBlue,citecolor=DarkGreen]{hyperref}
\usepackage{amsthm}
\usepackage{amscd}
\usepackage{geometry}
\usepackage{mathrsfs}
\usepackage{dsfont}

\newtheorem{thm}{Theorem}

\newtheorem{prop}{Proposition}
\newtheorem{lem}[prop]{Lemma}

\theoremstyle{definition}

\newtheorem*{rem}{Remark}

\newcommand{\leqs}{\leqslant }
\newcommand{\geqs}{\geqslant }
\allowdisplaybreaks

\begin{document}
\title{Upper bounds for $L^q$ norms of Dirichlet polynomials with small $q$}
\author{Winston Heap}
\thanks{Research supported by European Research Council grant no. 670239}
\email{winstonheap@gmail.com}
\maketitle
\begin{abstract}We improve on previous upper bounds for the $q$th norm of the partial sums of the Riemann zeta function on the half line when $0<q\leqs 1$. In particular, we show that the 1--norm is bounded above by $(\log N)^{1/4}(\log\log N)^{1/4}$. 
\end{abstract}

\section{Introduction and statement of results}

The $L^q$ norm of a Dirichlet polynomial $F(s)=\sum_{n\leqs N}a(n)n^{-s}$ is defined as
\[\|F\|_q^q=\lim_{T\to\infty}\frac{1}{T}\int_0^T |F(it)|^qdt\]
for $0<q<\infty$. 
In this note we are interested in the norms of the Dirichlet polynomial 
\[f(s)=\sum_{n\leqs N}\frac{1}{n^{1/2+s}}.\]
The norms of this function share similarities with moments of the Riemann zeta function on the half-line, being sometimes referred to as pseudomoments of the Riemann zeta function \cite{CG, BBSSZ}. In the paper \cite{CG}, Conrey and Gamburd showed that for $k\in\mathbb{N}$ 
\begin{equation}\label{2kth norm}
\|f\|_{2k}^{2k}\sim c_k (\log N)^{k^2}
\end{equation}
as $N\to\infty$ for a specific constant $c_k$. This of course bears a strong resemblance to the Keating--Snaith conjecture for the Riemann zeta function \cite{KS} and in fact, the proof of \eqref{2kth norm} employed random matrix theory.   

For the continuous case $0<q<\infty$, the norms of $f$ were investigated by Bondarenko, the author, and Seip in \cite{BHS}. It was shown that when $q>1$, we have the same order of magnitude as Conrey and Gamburd's result i.e.\, 
\begin{equation}\label{q>1 norms}\|f\|_q\asymp(\log N)^{q/4},\qquad q>1
\end{equation}
and that the lower bound here holds for all $q>0$, that is
\begin{equation*}\|f\|_q\gg(\log N)^{q/4},\qquad q>0.
\end{equation*} 

The situation regarding upper bounds for $q\leqs 1$ is less satisfactory and somewhat more interesting. Here, the results of \cite{BHS} gave
\begin{equation}\label{small norms}
\|f\|_q\ll\begin{cases}(\log N)^{1/4}\log\log N\,\,\,& q=1\\(\log N)^{1/4}\,\,\,& q<1.\end{cases}
\end{equation}  
The proof of these upper bounds relied on estimates for the norm of the partial sum operator $S_N$ which is defined by 
\[S_N\bigg(\sum_{n=1}^\infty a(n)n^{-s}\bigg)=\sum_{n\leqs N}a(n)n^{-s}.\]
On applying a generalisation of M. Riesz' Theorem due to Helson \cite{Helson}, it can be shown that for Dirichlet series $F(s)$ in $\mathscr{H}^q$ we have
\begin{equation}\label{S_N bounds}
\|S_NF\|_q\leqs\begin{cases}A_q\|F\|_q& q>1,\\B_q\|F\|_1\,\,\,& 0<q<1\end{cases}
\end{equation}
for some absolute constants $A_q$, $B_q$ (see Lemma 3 of \cite{BHS} and also 8.7.2 and 8.7.6 of \cite{R}). Here, the constants $A_q$, $B_q$ blow up like $|q-1|^{-1}$ as $q\to 1$. 

We write $f(s)=S_N(P_N(s))$ where $P_N(s)=\prod_{p\leqs N}(1-p^{-1/2-s})^{-1}$ so that $\|f\|_q=\|S_NP_N\|_q$. Since $\|P_N\|_q\asymp (\log N)^{q/4}$ for all $q>0$ (see Proposition \ref{easy prop} below), the bound in \eqref{small norms} for $q<1$ follows immediately by \eqref{S_N bounds}. The bound for the 1-norm follows from observing that $\|f\|_1\leqs \|f\|_q$ with $q=1+1/\log \log N$ and noting that 
\[\|S_NP_N\|_q\leqs A_q \|P_N\|_q\ll (q-1)^{-1}(\log N)^{q/4}\ll (\log N)^{1/4} \log\log N.\]

The bounds in \eqref{S_N bounds} apply to general Dirichlet series. However, in the special and most interesting case where all the coefficients of the Dirichlet series are 1, we can do better. 

\begin{thm}
\label{1 norm thm}
 There exists a constant $C_q>0$ such that for $0< q\leqs 1$, 
\begin{equation}\|f\|_q\leqs(1+o(1))C_q(\log N)^{\alpha_q}(\log\log N)^{\frac{1}{2}-\alpha_q}
\end{equation}
where
\[\alpha_q=\frac{1}{4(2-q)}.\] 
In particular
\begin{equation}\|f\|_1 \ll (\log N)^{1/4}(\log \log N)^{1/4}.
\end{equation}
\end{thm} 

\begin{rem}The constant $C_q$ can in fact be computed explicitly. We provide an example in equation \eqref{1 constant} below. 
\end{rem}

The proof of the above theorem uses the fact that we expect $f(s)$ to be well approximated by the Euler product $P_N(s)$ in the mean. Under this assumption, we can then mollify $f(s)$ appropriately. The theorem will follow from an application of H\"older's inequality along with the following two propositions. 

\begin{prop}\label{easy prop}Let $\beta \geqs 0$ and let 
\[J(Y,\beta)=\lim_{T\to\infty}\frac{1}{T}\int_0^T \bigg|\prod_{p\leqs Y}\bigg(1-\frac{1}{p^{1/2+it}}\bigg)^{-1}\bigg|^\beta dt.\]
Then \[J(Y,\beta)\sim a_\beta (e^\gamma\log Y)^{\frac{\beta^2}{4}}\]
as $Y\to\infty$ where 
\begin{equation}
\label{c_beta}
a_\beta=\prod_{p}\bigg(1-\frac{1}{p}\bigg)^{\beta^2/4}\sum_{m=0}^\infty \frac{d_{\beta/2}(p^m)^2}{p^m}
\end{equation} 
\end{prop}

\begin{prop}
\label{not so easy prop}
Let $a_\beta$ be as above,
\begin{equation}
\label{Y bounds}
\exp((\log\log N)^2)\leqs Y\leqs \exp\bigg(\frac{\log N}{5\log\log N}\bigg),
\end{equation}
and let
\[I(N,Y)=\lim_{T\to\infty}\frac{1}{T}\int_0^T \bigg|\sum_{n\leqs N}\frac{1}{n^{1/2+it}}\prod_{p\leqs Y}\bigg(1-\frac{1}{p^{1/2+it}}\bigg)^{1/2}\bigg|^2dt.\]
Then
\begin{equation}
\label{I bound}
I(N,Y)=a_1(e^\gamma\log Y)^{1/4}\bigg(\frac{\log N}{e^\gamma\log Y}+\frac{1}{\pi}\log\log Y+O(1)\bigg).
\end{equation}
\end{prop}

On taking $Y=\exp(\log N/5\log\log N)$ and applying H\"older's inequality in the form
\[\|f\|_q^q\leqs I(N,Y)^{\frac{q}{2}}J(Y, q/(2-q))^{1-\frac{q}{2}}\]
gives Theorem \ref{1 norm thm}. If we assume that Proposition \ref{not so easy prop} holds in the range 
\[Y\leqs \exp(\log N/B\log\log N)\] for some general $B>0$ then we find that 
\begin{equation}
\label{1 constant}
\|f\|_1\leqs (1+o(1))a_1e^{\gamma/4}\Big(\sqrt{B}e^{-\gamma}+\frac{1}{\pi\sqrt{B}}\Big)^{1/2}(\log N)^{1/4}(\log\log N)^{1/4}.
\end{equation}
We note that this is minimised when $B=e^\gamma/\pi$, and the constant becomes $(4/\pi)^{1/4}a_1$. Our restriction of $B\geqs 5$ could likely be improved. 

One may ask what the true order of $\|f\|_q$ is for $0<q\leqs 1$. It might be reasonable to expect that $\|f\|_q\ll (\log N)^{q/4}$ in the full range of $q$. Certainly, this would be in analogy with the moments of the Riemann zeta function as proved conditionally on the Riemann hypothesis by Harper \cite{H2}, extending the work of Soundarajan \cite{S}. However, in a recent paper of Bondarenko et. \!al \cite{BBSSZ} (Theorem 6.2) it was shown that for $\alpha>1$ and small $q$, the $q$th norm of the Dirichlet polynomial $\sum_{n\leqs N} d_\alpha(n)n^{-1/2-s}$ exceeds the upper bound of $(\log N)^{\alpha^2q/4}$ that one might expect. Thus, the analogy with the Riemann zeta function is not as robust as expected, at least for small $q$. 

Very recently, Harper \cite{H1} has obtained sharp bounds on the $q$th norm of the Dirichlet polynomial $\sum_{n\leqs N}n^{-it}$ when $0\leqs q\leqs 1$. In particular he has shown that 
\[\lim_{T\to\infty}\frac{1}{T}\int_0^T\Big|\sum_{n\leqs N}n^{-it}\Big|dt\asymp \frac{\sqrt{N}}{(\log\log N)^{1/4}},\] and in doing so has proved a conjecture of Helson.  It would be interesting to see if his methods may be applied to our situation in order to answer the question on the true growth of $\| f\|_q$.

\section{Proof of the Propositions}

Throughout we let $S(Y)$ denote the set of $Y$-- smooth numbers, that is, the set of positive integers all of whose prime divisors are less than or equal to $Y$. 
%\[S(Y)=\Big\{n\in\mathbb{N}:p|n\implies p\leqs Y\Big\}\]
The proof of Proposition \ref{easy prop} is relatively easy so we give that first. 

\begin{proof}[Proof of Proposition \ref{easy prop}] For $\alpha\in\mathbb{R}$ let $d_\alpha(n)$ be the coefficients of $\zeta(s)^\alpha$ so that, in particular, $d_\alpha(p)=\alpha$ for all primes $p$. Then we have
\begin{equation*}
\begin{split}
&\lim_{T\to\infty}\frac{1}{T}\int_0^T \bigg|\prod_{p\leqs Y}\bigg(1-\frac{1}{p^{1/2+it}}\bigg)^{-1}\bigg|^\beta dt
=\lim_{T\to\infty}\frac{1}{T}\int_0^T \bigg|\sum_{n\in S(Y)}\frac{d_{\beta/2}(n)}{n^{1/2+it}}\bigg|^2dt.
\end{split}
\end{equation*}
By Carlson's Theorem \cite{SS},  
\begin{equation*}
\begin{split}
\lim_{T\to\infty}\frac{1}{T}\int_0^T \bigg|\sum_{n\in S(Y)}\frac{d_{\beta/2}(n)}{n^{1/2+it}}\bigg|^2dt=&\sum_{n\in S(Y)} \frac{d_{\beta/2}(n)^2}{n}.
\end{split}
\end{equation*}
Alternatively, one may argue this by first truncating the sum at $n=\log T$, say, and then applying the Montgomery--Vaughan mean value Theorem  (see Lemma 2 of \cite{GHK} for example). 
 
 Now,
 \[\sum_{n\in S(Y)} \frac{d_{\beta/2}(n)^2}{n}=\prod_{p\leqs Y}\sum_{m=0}^\infty\frac{d_{\beta/2}(p^m)^2}{p^m}=C_\beta(Y)\prod_{p\leqs Y}\bigg(1-\frac{1}{p}\bigg)^{-\beta^2/4}\]
where 
\begin{equation*}
\begin{split}
C_\beta(Y)=&\prod_{p\leqs Y}\bigg(1-\frac{1}{p}\bigg)^{\beta^2/4}\sum_{m=0}^\infty \frac{d_{\beta/2}(p^m)^2}{p^m}\\
=&\prod_{p}\bigg(1-\frac{1}{p}\bigg)^{\beta^2/4}\sum_{m=0}^\infty \frac{d_{\beta/2}(p^m)^2}{p^m}\prod_{p>Y}\Big(1+O(p^{-2})\Big)\\
=&a_\beta\bigg(1+O\Big(\frac{1}{Y\log Y}\Big)\bigg)
\end{split}
\end{equation*}
The result now follows by Mertens' Theorem.
\end{proof}

For the proof of Proposition \ref{not so easy prop} we start by noting that
\begin{equation}\label{I sum}
\begin{split}
I(N,Y)
=&\lim_{T\to\infty}\frac{1}{T}\int_0^T \bigg|\sum_{n\leqs N}\frac{1}{n^{1/2+it}}\prod_{p\leqs Y}\bigg(1-\frac{1}{p^{1/2+it}}\bigg)^{1/2}\bigg|^2dt\\
=&\lim_{T\to\infty}\frac{1}{T}\int_0^T \bigg|\sum_{n\leqs N}\frac{1}{n^{1/2+it}}\sum_{m\in S(Y)}\frac{d_{-1/2}(m)}{n^{1/2+it}}\bigg|^2dt\\
=&\sum_{\substack{m_1n_1=m_2n_2\\m_j\in S(Y)\\n_j\leqs N}}\frac{d_{-1/2}(m_1)d_{-1/2}(m_2)}{(m_1n_1m_2n_2)^{1/2}}
\end{split}
\end{equation}
where we have again used Carlson's Theorem in the last line. 
We wish to apply Perron's formula in the variables $n_1$ and $n_2$ of this last sum. Consequently, we first look at the associated Dirichlet series.

\begin{lem}
\label{F lem}
Let
\[F(s_1,s_2,Y)=\sum_{\substack{m_1n_1=m_2n_2\\m_j\in S(Y)}}\frac{d_{-1/2}(m_1)d_{-1/2}(m_2)}{(m_1m_2)^{1/2}n_1^{1/2+s_1}n_2^{1/2+s_2}}.\]
Then
\begin{equation*}
\begin{split}
F(s_1,s_2,Y)=A(\underline{s})\zeta(1+s_1+s_2)\prod_{p\leqs Y}\bigg(1-\frac{1}{p}\bigg)^{-1/4}\bigg(1-\frac{1}{p^{1+s_1}}\bigg)^{1/2}\bigg(1-\frac{1}{p^{1+s_2}}\bigg)^{1/2}
\end{split}
\end{equation*}
where $A(\underline{s})=A(s_1,s_2,Y)$ is some product that is absolutely convergent (in the limit $Y\to\infty$) in the regions $\Re(s_1),\Re(s_2)>-1/4$. 
\end{lem}
\begin{proof}Since the condition $m_1n_1=m_2n_2$ is multiplicative and so are the coefficients of our sum, we may expand it as an Euler product to see that it is
\begin{equation*}
\begin{split}
&\prod_{p\leqs Y}\sum_{\substack{e_1+f_1=\\e_2+f_2}}\frac{d_{-1/2}(p^{e_1})d_{-1/2}(p^{e_2})}{p^{\frac{1}{2}e_1+\frac{1}{2}e_2+(\frac{1}{2}+s_1)f_1+(\frac{1}{2}+s_2)f_2}}\prod_{p> Y}\sum_{f_1=f_2}\frac{1}{p^{(\frac{1}{2}+s_1)f_1+(\frac{1}{2}+s_2)f_2}}\\
=&\prod_{p\leqs Y}\bigg(1+\frac{1}{4p}-\frac{1}{2p^{1+s_1}}-\frac{1}{2p^{1+s_2}}+\frac{1}{p^{1+s_1+s_2}}+\cdots\bigg)\prod_{p>Y}\bigg(1-\frac{1}{p^{1+s_1+s_2}}\bigg)^{-1}\\
=&A(\underline{s})\zeta(1+s_1+s_2)\prod_{p\leqs Y}\bigg(1-\frac{1}{p}\bigg)^{-1/4}\bigg(1-\frac{1}{p^{1+s_1}}\bigg)^{1/2}\bigg(1-\frac{1}{p^{1+s_2}}\bigg)^{1/2}.
\end{split}
\end{equation*}
Explicitly, we have
\begin{equation}\label{arithmetic factor}
A(\underline{s})=\prod_{p\leqs Y}\frac{(1-p^{-1})^{1/4}(1-p^{-1-s_1-s_2})}{(1-p^{-1-s_1})^{1/2}(1-p^{-1-s_2})^{1/2}}\sum_{\substack{e_1+f_1=\\e_2+f_2}}\frac{d_{-1/2}(p^{e_1})d_{-1/2}(p^{e_2})}{p^{\frac{1}{2}e_1+\frac{1}{2}e_2+(\frac{1}{2}+s_1)f_1+(\frac{1}{2}+s_2)f_2}}
\end{equation}
and, after expanding everything, this is seen to be absolutely convergent (as $Y\to\infty$) in the stated region.  
\end{proof}

As usual, there are some extra technicalities in Perron's formula when $N$ is an integer.  To deal with these, note that the sum in \eqref{I sum} is unaffected if we replace $N$ by the nearest half integer, taking $N+1/2$ if $N\in\mathbb{N}$. Since the asymptotic formula in \eqref{I bound} is also unaffected by such a change, we henceforth assume that $N$ is a half integer.

\begin{lem}\label{perron trunc lem}Let $F(s_1,s_2,Y)$ be as in the above lemma.  
Then for $Y\geqs 2$ and large $T_1\leqs T_2\leqs N$ we have
\begin{multline}\label{I_1 trunc}
I(N,Y)=\frac{1}{(2\pi i)^2}\int_{c_1-iT_1}^{c_1+iT_1}\int_{c_2-iT_2}^{c_2+iT_2}F(s_1,s_2,Y)N^{s_1+s_2}\frac{ds_2}{s_2}\frac{ds_1}{s_1}\\
+O\bigg(\frac{(\log Y)^{2}(\log N)^3}{T_1^{1/2}}\bigg)
\end{multline} 
where $c_1,\,c_2\asymp 1/\log N$. 
\end{lem}

\begin{proof}

Let $\chi_{[a,b]}(t)$ denote the characteristic function of the interval $[a,b]$. Then for $n\in\mathbb{N}$ 
\[\chi_{[1,N]}(n)=\frac{1}{2\pi i}\int_{c-iT}^{c+iT}\bigg(\frac{N}{n}\bigg)^s\frac{ds}{s}+O\bigg(\frac{(N/n)^c}{\max(T|\log(N/n)|,1)}\bigg)\]
for any $c>0$. See the first Lemma of chapter 17 of Davenport's book \cite{D} for example.

Applying this to the variables $n_1$ and $ n_2$ in formula \eqref{I sum} with the choice $c_1,\,c_2\asymp 1/\log N$ gives
\begin{multline}
\label{perron trunc 1}
I(N,Y)=\sum_{\substack{m_1n_1=m_2n_2\\m_j\in S(Y)}}\frac{d_{-1/2}(m_1)d_{-1/2}(m_2)}{(m_1n_1m_2n_2)^{1/2}}\\
\times\bigg[\frac{1}{2\pi i}\int_{c_1-iT_1}^{c_1+iT_1}\bigg(\frac{N}{n_1}\bigg)^{s_1}\frac{ds_1}{s_1}+O\bigg(\frac{n_1^{-c_1}}{\max(T_1|\log(N/n_1)|,1)}\bigg)\bigg]\\
\times\bigg[\frac{1}{2\pi i}\int_{c_2-iT_2}^{c_2+iT_2}\bigg(\frac{N}{n_2}\bigg)^{s_2}\frac{ds_2}{s_2}+O\bigg(\frac{n_2^{-c_2}}{\max(T_2|\log(N/n_2)|,1)}\bigg)\bigg].
\end{multline}
After swapping the order of summation and integration, the term arising from the product of the two integrals is seen to be the integral in equation \eqref{I_1 trunc}. This swapping is legitimate since by Lemma \ref{F lem} we are in the domain of absolute convergence of $F(s_1,s_2,Y)$. It remains to estimate the error terms.

Let's first consider a term arising from the product of an integral and an error term. Note that the integrals are
\begin{equation*}
\begin{split}
\frac{1}{2\pi i}\int_{c_j-iT_j}^{c_j+iT_j}\bigg(\frac{N}{n_j}\bigg)^{s_j}\frac{ds_j}{s_j}
\ll &n_j^{-c_j}\int_{-T}^T\frac{dt}{|c_j+it|}=n_j^{-c_j}\bigg(\int_{-T}^{-1}+\int_{-1}^1+\int_1^T\bigg)\frac{dt}{|c_j+it|} \\
\ll& n_j^{-c_j}(\log N+\log T)\\
\ll& n_j^{-c_j}\log N.
\end{split}
\end{equation*}
Hence, the product of an integral and an error term is 
\[\ll \log N\sum_{\substack{m_1n_1=m_2n_2\\m_j\in S(Y)}}\frac{|d_{-1/2}(m_1)d_{-1/2}(m_2)|}{(m_1m_2)^{1/2}n_1^{1/2+c_1}n_2^{1/2+c_2}\max(T_1|\log(N/n_1)|,1)}.\] 
We first apply the bound $|d_{-1/2}(m)|\leqs 1$.  

Now, for the terms with $|\log (N/n_1)|\geqs T_1^{-1/2}$ we have
\begin{align}\label{perron error 1}
& \log N\sum_{\substack{m_1n_1=m_2n_2\\m_j\in S(Y)\\|\log (N/n_1)|\geqs T^{-1/2}}}\frac{1}{(m_1m_2)^{1/2}n_1^{1/2+c_1}n_2^{1/2+c_2}\max(T_1|\log(N/n_1)|,1)}\\
&\leqs \frac{\log N}{T_1^{1/2}}\sum_{\substack{m_1n_1=m_2n_2\\m_j\in S(Y)}}\frac{1}{(m_1m_2)^{1/2}n_1^{1/2+c_1}n_2^{1/2+c_2}}\nonumber\\
&= \frac{\log N}{T_1^{1/2}}\sum_{\substack{m_1n_1=l\\m_1\in S(Y)}}\frac{\sum_{m_2n_2=l,m_2\in S(Y)}n_2^{-c_2}}{m_1^{1/2}n_1^{1/2+c_1}l^{1/2}}\nonumber\\
&\leqs \frac{\log N}{T_1^{1/2}}\sum_{\substack{m_1n_1=l\\ m_1\in S(Y)}}\frac{d(l)}{m_1^{1/2}n_1^{1/2+c_1}l^{1/2}}\nonumber\\
&\leqs \frac{\log N}{T_1^{1/2}}\sum_{\substack{m_1,n_1\\ m_1\in S(Y)}}\frac{d(m_1)d(n_1)}{m_1n_1^{1+c_1}}\nonumber\\
&\ll \frac{(\log Y)^2(\log N)^3}{T_1^{1/2}}\nonumber
\end{align}
where we have used the inequality $d(mn)\leqs d(m)d(n)$.

For the terms with $|\log (N/n_1)|\leqs T_1^{-1/2}$ we have $|n_1-N|\leqs ANT_1^{-1/2}$ where $A$ is some positive constant. We write $n_1=N+r$ with $|r|\leqs ANT_1^{-1/2}$. Then
\begin{equation*}
\begin{split}
& \log N\sum_{\substack{m_1n_1=m_2n_2\\m_j\in S(Y)\\|\log (N/n_1)|\leqs T^{-1/2}}}\frac{1}{(m_1m_2)^{1/2}n_1^{1/2+c_1}n_2^{1/2+c_2}\max(T_1|\log(N/n_1)|,1)}\\
&\ll \log N \sum_{|r|\leqs ANT^{-1/2}}\sum_{\substack{m_1(N+r)=m_2n_2\\m_j\in S(Y)}}\frac{1}{(m_1m_2)^{1/2}(N+r)^{1/2+c_1}n_2^{1/2+c_2}}\\
&=\log N\sum_{|r|\leqs ANT^{-1/2}}\sum_{\substack{m_1(N+r)=l\\m_1\in S(Y)}}\frac{\sum_{m_2n_2=l,m_2\in S(Y)}n_2^{-c_2}}{m_1^{1/2}(N+r)^{1/2+c_1}l^{1/2}}\\
&\leqs \log N\sum_{|r|\leqs ANT^{-1/2}}\sum_{\substack{m_1\in S(Y)}}\frac{d(m_1(N+r))}{m_1(N+r)^{1+c_1}}\\
&\ll (\log N)(\log Y)^2 \sum_{|r|\leqs ANT_1^{-1/2}}\frac{d(N+r)}{(N+r)^{1+c_1}}.
\end{split}
\end{equation*}
Now,
\begin{equation*}
\begin{split}
\sum_{|r|\leqs ANT_1^{-1/2}}\frac{d(N+r)}{N+r}=&\sum_{N(1-AT_1^{-1/2})\leqs n\leqs N(1+AT_1^{-1/2})}\frac{d(n)}{n}
\end{split}
\end{equation*}
which is a short logarithmic average, and hence should be small. Indeed, on applying the formula
\[\sum_{n\leqs X}\frac{d(n)}{n}=\frac{1}{2}\log^2 X+c_1\log X +c_0+O(X^{-1/2})\]
we see that 
\begin{equation*}
\begin{split}
\sum_{N(1-AT_1^{-1/2})\leqs n\leqs N(1+AT_1^{-1/2})}\frac{d(n)}{n}\ll \frac{\log N}{T_1^{1/2}}.
\end{split}
\end{equation*}
Applying this gives an error term that is $\ll(\log Y)^2(\log N)^2T_1^{-1/2}$. 
To sum up: a term arising from the product of an integral and an error term in \eqref{perron trunc 1} is
\[\ll \frac{(\log Y)^2(\log N)^3}{T_1^{1/2}}.\]

The product of the two error terms in \eqref{perron trunc 1} is given by the sum
\begin{equation*}
\begin{split}
&\sum_{\substack{m_1n_1=m_2n_2\\m_j\in S(Y)}}\frac{|d_{-1/2}(m_1)d_{-1/2}(m_2)|}{(m_1m_2)^{1/2}n_1^{1/2+c_1}n_2^{1/2+c_2}\max(T_1|\log(N/n_1)|,1)\max(T_2|\log(N/n_2)|,1)}\\
&\leqs\sum_{\substack{m_1n_1=m_2n_2\\m_j\in S(Y)}}\frac{|d_{-1/2}(m_1)d_{-1/2}(m_2)|}{(m_1m_2)^{1/2}n_1^{1/2+c_1}n_2^{1/2+c_2}\max(T_1|\log(N/n_1)|,1)}
\end{split}
\end{equation*}
and as we have already seen, this is $\ll T_1^{-1/2}(\log N)^2(\log Y)^2$. 
\end{proof}

We take 
\[T_1=T=\exp((\log\log N)^2)\]
and $T_2=2T_1$ in Lemma \ref{perron trunc lem} so that the error term in \eqref{I_1 trunc} is $o(1)$. Hence we may concentrate on the integral
\[\frac{1}{(2\pi i)^2}\int_{c_1-iT}^{c_1+iT}\int_{c_2-2iT}^{c_2+2iT}F(s_1,s_2,Y)N^{s_1+s_2}\frac{ds_2}{s_2}\frac{ds_1}{s_1}\]
where we recall that 
\begin{equation*}F(s_1,s_2,Y)=\bigg[\prod_{p\leqs Y}\bigg(1-\frac{1}{p}\bigg)^{-1/4}\bigg]A(\underline{s})\zeta(1+s_1+s_2)\prod_{p\leqs Y}\bigg(1-\frac{1}{p^{1+s_1}}\bigg)^{1/2}\bigg(1-\frac{1}{p^{1+s_2}}\bigg)^{1/2}.
\end{equation*}

We first shift the contour in $s_2$ to the line with real part $-\kappa_2=-1/\log Y$. We pick up two simple poles; one at $s_2=0$ and the other at $s_2=-s_1$. As usual, the horizontal contours will give a negligible contribution. With our choice of $\kappa_2$ and the upper bound $Y\leqs \exp(\log N/(5\log\log N))$ we have $N^{-\kappa_2}\leqs (\log N)^{-5}$ and so the vertical contour on the left will also be seen to be negligible. This is the content of the next lemma. 

\begin{lem}
\label{s_2 shift}
Let $T=\exp((\log\log N)^2)$ and let $Y$ satisfy the bounds in \eqref{Y bounds}. Then
\begin{multline}\label{I_2 first shift}
I(N,Y)=\bigg[\prod_{p\leqs Y}\bigg(1-\frac{1}{p}\bigg)^{-1/4}\bigg]\frac{1}{2\pi i}\int_{c_1-iT}^{c_1+iT}\bigg(G(s_1,Y)+H(s_1,Y)\bigg)N^{s_1}\frac{ds_1}{s_1}\\
+O\big((\log N)^{-3/4}\big)
\end{multline}
where
\begin{equation}\label{G}
G(s_1,Y)=\bigg[\prod_{p\leqs Y}\bigg(1-\frac{1}{p}\bigg)^{1/2}\bigg]A(s_1,0,Y)\zeta(1+s_1)\prod_{p\leqs Y}\bigg(1-\frac{1}{p^{1+s_1}}\bigg)^{1/2}
\end{equation}
and
\begin{equation}\label{H}
H(s_1,Y)=-\frac{N^{-s_1}}{s_1}A(s_1,-s_1,Y)\prod_{p\leqs Y}\bigg(1-\frac{1}{p^{1+s_1}}\bigg)^{1/2}\bigg(1-\frac{1}{p^{1-s_1}}\bigg)^{1/2}
\end{equation}
\end{lem}  
\begin{proof}We consider the integral over $s_2$ as part of the rectangular contour with vertices $c_2-2iT,\,c_2+2iT,\,-\kappa_2+2iT,\,-\kappa_2-2iT$ with $\kappa_2=1/\log Y$. This contour encloses $-s_1$ since $\kappa_2\gg c_1$. Hence,
\begin{equation*}
\begin{split}
&\frac{1}{(2\pi i)^2}\int_{c_1-iT}^{c_1+iT}\int_{c_2-2iT}^{c_2+2iT}F(s_1,s_2,Y)N^{s_1+s_2}\frac{ds_2}{s_2}\frac{ds_1}{s_1}\\
=&\frac{1}{2\pi i}\int_{c_1-iT}^{c_1+iT}\bigg[\mathrm{Res}_{s_2=0}\Big(F(s_1,s_2,Y)\frac{N^s_2}{s_2}\Big)+\mathrm{Res}_{s_2=-s_1}\Big(F(s_1,s_2,Y)\frac{N^s_2}{s_2}\Big)\\
&-\frac{1}{2\pi i}\Big(\int_{c_2+2iT}^{-\kappa_2+2iT}+\int_{-\kappa_2+2iT}^{-\kappa_2-2iT}+\int_{-\kappa_2-2iT}^{c_2-2iT}\Big)F(s_1,s_2,Y)N^{s_2}\frac{ds_2}{s_2}\bigg]N^{s_1}\frac{ds_1}{s_1}.
\end{split}
\end{equation*}
The residue terms at $s_2=0$ and $s_2=-s_1$ give rise to the terms $G(s_1,Y)$ and $H(s_1,Y)$  respectively. Thus it remains to estimate the last three integrals. 

For $s_2$ on any of these contours we have the bounds
\[A(s_1,s_2,Y)\ll 1\]
and
\begin{equation*}
\begin{split}
\prod_{p\leqs Y}&\bigg(1-\frac{1}{p}\bigg)^{-1/4}\bigg(1-\frac{1}{p^{1+s_1}}\bigg)^{1/2}\bigg(1-\frac{1}{p^{1+s_2}}\bigg)^{1/2}\\
\ll&\prod_{p\leqs Y}\bigg(1-\frac{1}{p}\bigg)^{-1/4}\bigg(1+\frac{1}{p^{1+\sigma_1}}\bigg)^{1/2}\bigg(1+\frac{1}{p^{1+\sigma_2}}\bigg)^{1/2}\\
\ll& (\log Y)^{5/4}
\end{split}
\end{equation*} 
since $\sigma_1\asymp1/\log N\ll1/\log Y$ and $\sigma_2\ll 1/\log Y$. Therefore, on the regions of integration under consideration we have
\begin{equation}
\label{F bound}
F(s_1,s_2,Y)\ll |\zeta(1+s_1+s_2)|(\log Y)^{5/4}.
\end{equation}

Consider the integral involving the upper horizontal section. Applying the above bound for $F$ we see that this integral is
\begin{equation*}
\begin{split}\ll&(\log Y)^{5/4}\int_{-T}^{T}\int_{c_2}^{-\kappa_2}|\zeta(1+c_1+\sigma_2+it_1+2iT)|N^{c_1+\sigma_2}\frac{dt_1}{|c_1+it_1|}\frac{d\sigma_2}{|\sigma_2+2iT|}\\
\ll&\frac{(\log Y)^{5/4}}{T}(\log N+\log T)\max_{\substack{\sigma_2\in[-\kappa_2,c_2]\\t_1\in[-T,T]}}|\zeta(1+c_1+\sigma_2+it_1+2iT)|
\end{split}
\end{equation*}
with this second factor arising from splitting the $t_1$-integral over the regions $t_1\ll 1$ and $1\ll |t_1|\leqs T$. 
Since $t_1+2T\leqs 3T= 3\exp((\log\log N)^2)$ we have
\[1+c_1+\sigma_2\geqs1-\frac{1}{\log Y}\geqs 1-\frac{1}{(\log \log N)^2}\geqs 1-\frac{c}{\log(t_1+2T)}\]
and since $t_1+2T\geqs T\gg 1$ we may apply the usual bounds (see Theorem 5.17 of Titchmarsh \cite{T} for example)  to give
\begin{equation}\label{zeta bound}
\max_{\substack{\sigma_2\in[-\kappa_2,c_2]\\t_1\in[-T,T]}}|\zeta(1+c_1+\sigma_2+it_1+2iT)|\ll (\log 3T)^5.
\end{equation}
A similar bound for the lower horizontal integral applies. 

For the remaining vertical integral we again apply the bound for $F$ given in \eqref{F bound}. We then have something that is  
\begin{equation*}
\begin{split}
\ll (\log Y )^{5/4} \int_{-T}^T\int_{-2T}^{2T} |\zeta(1+c_1-\kappa_2+i(t_1+t_2))|N^{c_1-\kappa_2}\frac{dt_1}{|c_1+it_1|}\frac{dt_2}{|-\kappa_2+it_2|}.
\end{split}
\end{equation*}  
If $|t_1+t_2|\ll 1$ then we may bound the zeta function by $1/|c_1-\kappa_2|\ll\log Y$ and in the other case we may use the bound \eqref{zeta bound}. Therefore the above integral is
\begin{equation*}
\begin{split}
\ll &(\log Y)^{5/4} \max(\log Y, (\log T)^5)(\log N+\log T)(\log Y+\log T) N^{-1/\log Y}\\
\ll&(\log N)^{17/4}(\log N)^{-5}\\
\leqs&(\log N)^{- 3/4}
\end{split}
\end{equation*}
\end{proof}

The integral involving $H(s_1,Y)$ is somewhat trickier so we postpone that till later. As for the integral involving $G(s_1,Y)$, we can  compute it directly with residues to give the following. 

\begin{lem}\label{G lem}Let $G(s,Y)$ be given by \eqref{G}. Then 
\begin{equation}
\frac{1}{2\pi i}\int_{c_1-iT}^{c_1+iT}G(s,Y)N^{s}\frac{ds}{s}=a_1e^{-\gamma}\bigg(\frac{\log N}{\log Y}+\frac{1}{2}\bigg)+O(1/\log Y)
\end{equation}
where $a_1$ is given by \eqref{c_beta}. 
\end{lem}

\begin{proof}
Applying the definition of $G$ we have 
\begin{multline*}
\frac{1}{2\pi i}\int_{c_1-iT}^{c_1+iT}G(s,Y)N^{s}\frac{ds}{s}\\
=\bigg[\prod_{p\leqs Y}\bigg(1-\frac{1}{p}\bigg)^{1/2}\bigg]\frac{1}{2\pi i}\int_{c_1-iT}^{c_1+iT}A(s,0,Y)\zeta(1+s)\prod_{p\leqs Y}\bigg(1-\frac{1}{p^{1+s}}\bigg)^{1/2}N^s\frac{ds}{s}.
\end{multline*}
Again, we shift this contour to the line with real part $-\kappa_1=-1/\log Y$ and encounter a double pole at $s=0$. This gives
\begin{equation*}
\begin{split}
&\frac{1}{2\pi i}\int_{c_1-iT}^{c_1+iT}A(s,0,Y)\zeta(1+s)\prod_{p\leqs Y}\bigg(1-\frac{1}{p^{1+s}}\bigg)^{1/2}N^s\frac{ds}{s}\\
=&\mathrm{Res}_{s=0}\bigg[A(s,0,Y)\zeta(1+s)\prod_{p\leqs Y}\bigg(1-\frac{1}{p^{1+s}}\bigg)^{1/2}\frac{N^s}{s}\bigg]\\
&-\frac{1}{2\pi i}\bigg(\int_{c_1+iT}^{-\kappa_1+iT}+\int_{-\kappa_1+iT}^{-\kappa_1-iT}+\int_{-\kappa_1-iT}^{c_1-iT}\bigg)A(s,0,Y)\zeta(1+s)\prod_{p\leqs Y}\bigg(1-\frac{1}{p^{1+s}}\bigg)^{1/2}N^s\frac{ds}{s}.
\end{split}
\end{equation*} 
Applying the usual bounds, as in Lemma \ref{s_2 shift}, the horizontal contours are 
\[\ll \frac{(\log Y)^{1/2}(\log T)^{5}}{T}\]
and the vertical contour is 
\[\ll N^{-\kappa_1}(\log Y)^{5/2}\ll (\log Y)^{5/2}(\log N)^{-5}\ll (\log N)^{-5/2}.\]
The residue term is given by
\[A(0,0,Y)\prod_{p\leqs Y}\bigg(1-\frac{1}{p}\bigg)^{1/2}\bigg[\log N+\frac{1}{2}\sum_{p\leqs Y}\frac{\log p}{p-1}+\gamma+\frac{A^\prime(0,0,Y)}{A(0,0,Y)}\bigg].\]
Since 
\[\frac{1}{2}\sum_{p\leqs Y}\frac{\log p}{p-1}=\frac{1}{2}\log Y+O(1)\]
we have
\begin{multline*}
\frac{1}{2\pi i}\int_{c_1-iT}^{c_1+iT}G(s,Y)N^{s}\frac{ds}{s}=A(0,0,Y)\prod_{p\leqs Y}\bigg(1-\frac{1}{p}\bigg)\bigg[\log N+\frac{1}{2}\log Y+O(1)\bigg]\\+O\bigg(\prod_{p\leqs Y}\Big(1-\frac{1}{p}\Big)^{1/2}(\log N)^{-5/2}\bigg).
\end{multline*}
The result now follows on applying Mertens' Theorem and the formula
\begin{equation*}
\begin{split}
A(0,0,Y)=&\prod_p\frac{(1-p^{-1})^{1/4}(1-p^{-1})}{(1-p^{-1})^{1/2}(1-p^{-1})^{1/2}}\sum_{\substack{e_1+f_1=\\e_2+f_2}}\frac{d_{-1/2}(p^{e_1})d_{-1/2}(p^{e_2})}{p^{\frac{1}{2}(e_1+e_2+f_1+f_2)}}\\
&\times\prod_{p>Y}\big(1+O(p^{-2})\big)\\
=&a_1\bigg(1+O\Big(\frac{1}{Y\log Y}\Big)\bigg).
\end{split}
\end{equation*}

\end{proof}

By combining Lemmas \ref{s_2 shift} and \ref{G lem} we have the following formula for $I(N,Y)$:
\begin{equation*}
\begin{split}
I(N,Y)=\prod_{p\leqs Y}\bigg(1-\frac{1}{p}\bigg)^{-1/4}\bigg[a_1e^{-\gamma}\bigg(\frac{\log N}{\log Y}+\frac{1}{2}\bigg)+I_2(Y)\bigg]+O(1/(\log Y)^{3/4})
\end{split}
\end{equation*}
where 
\[I_2(Y)=-\frac{1}{2\pi i}\int_{c_1-iT}^{c_1+iT}A(s,-s,Y)\prod_{p\leqs Y}\bigg(1-\frac{1}{p^{1+s}}\bigg)^{1/2}\bigg(1-\frac{1}{p^{1-s}}\bigg)^{1/2}\frac{ds}{s^2}.\]
Thus, in order to prove Proposition \ref{not so easy prop} it suffices to show that $I_2(Y)= a_1\pi^{-1}\log\log Y+O(1)$. 

Note that shifting contours to the far left or right in $I_2(Y)$ will not help since the product over primes gets large in both directions. In any case, the residue at $s=0$ is zero since the integrand is an even function and hence has a zero derivative at $s=0$. Indeed, the integrand is small for $t\leqs 1/\log Y$ and hence the major contribution should occur away from $s=0$. With this in mind, we first shift the contour to the line with $\Re(s)=0$ to exploit some symmetry. Then for $t\geqs1/\log Y$ we may approximate the product over primes by the zeta function, which is allowed since the products are fairly long and we are on the edge of the critical strip. The major contribution is then seen to come from the region $1/\log Y\leqs t\leqs 1$.   

\begin{lem}We have
\[I_2(Y)=\frac{1}{\pi}\int_{1/\log Y}^T A(it,-it,Y)\bigg|\prod_{p\leqs Y}\bigg(1-\frac{1}{p^{1+it}}\bigg)\bigg|\frac{dt}{t^2}+O(1)\]
\end{lem}

\begin{proof}
We first shift the line of integration to $c=2/\log Y$ and incur a term that is $o(1)$ from the horizontal sections. We then consider $I_2(Y)$ as part of the rectangular contour with vertices $c-iT,\,c+iT,\,iT,\,-iT$ whose left edge has a semicircular indentation centered at $s=0$ of radius $1/\log Y$.    

Thus,
\begin{multline*}I_2(Y)=-\frac{1}{2\pi i}\bigg(\int_{-iT}^{-i/\log Y}+\int_{\substack{s=e^{i\theta}/\log Y\\-\frac{\pi}{2}\leqs \theta\leqs \frac{\pi}{2}}}+\int_{i/\log Y}^{iT}\bigg)A(s,-s,Y)\\
\times\prod_{p\leqs Y}\bigg(1-\frac{1}{p^{1+s}}\bigg)^{1/2}\bigg(1-\frac{1}{p^{1-s}}\bigg)^{1/2}\frac{ds}{s^2}+o(1).
\end{multline*}
The integral over the semicircular arc is seen to be $O(1)$ whilst
\begin{equation*}
\begin{split}&-\frac{1}{2\pi i}\int_{-iT}^{-i/\log Y}A(s,-s,Y)\prod_{p\leqs Y}\bigg(1-\frac{1}{p^{1+s}}\bigg)^{1/2}\bigg(1-\frac{1}{p^{1-s}}\bigg)^{1/2}\frac{ds}{s^2}\\
=&-\frac{1}{2\pi}\int_{-T}^{-1/\log Y}A(it,-it,Y)\bigg|\prod_{p\leqs Y}\bigg(1-\frac{1}{p^{1+it}}\bigg)\bigg|\frac{dt}{(it)^2}\\
=&-\frac{1}{2\pi}\int_{1/\log Y}^TA(-it,it,Y)\bigg|\prod_{p\leqs Y}\bigg(1-\frac{1}{p^{1+it}}\bigg)\bigg|\frac{dt}{(it)^2}\\
\end{split}
\end{equation*}
This is equal to the integral in the upper half plane provided that we can show $A(-it,it)=A(it,-it)$. Now, 
\[A(s_1,s_2,Y)=\prod_{p\leqs Y}\frac{(1-p^{-1})^{1/4}(1-p^{-1-s_1-s_2})}{(1-p^{-1-s_1})^{1/2}(1-p^{-1-s_2})^{1/2}}\sum_{\substack{e_1+f_1=\\e_2+f_2}}\frac{d_{-1/2}(p^{e_1})d_{-1/2}(p^{e_2})}{p^{\frac{1}{2}e_1+\frac{1}{2}e_2+(\frac{1}{2}+s_1)f_1+(\frac{1}{2}+s_2)f_2}}\]
and hence
\[A(-it,it,Y)=\prod_{p\leqs Y}\frac{(1-p^{-1})^{1/4}(1-p^{-1})}{|1-p^{-1-it}|}\sum_{\substack{m\geqs 0}}\Big|\sum_{n_1n_2=p^m}d_{-1/2}(n_1)n_2^{it}\Big|^2{p^{-m}}\]
and the result follows. Incidentally, this last formula shows that $A(it,-it,Y)$ is positive.
\end{proof}

\begin{lem}We have
\[I_2(Y)=\frac{a_1}{\pi} \log\log Y+O(1)\]
where $a_1$ is given by \eqref{c_beta}
\end{lem}

\begin{proof}
We first note that 
\begin{equation}\label{I_2}
\begin{split}
I_2(Y)=&\frac{1}{\pi}\int_{1/\log Y}^TA(it,-it,Y)\bigg|\prod_{p\leqs Y}\bigg(1-\frac{1}{p^{1+it}}\bigg)\bigg|\frac{dt}{t^2}+O(1)\\
=&\frac{1}{\pi}\int_{1/\log Y}^{\log Y}A(it,-it,Y)\bigg|\prod_{p\leqs Y}\bigg(1-\frac{1}{p^{1+it}}\bigg)\bigg|\frac{dt}{t^2}+O(1)
\end{split}
\end{equation}
since $A(it,-it,Y)\ll 1$ and the product over primes is $\ll \log Y$. 
Throughout, we consider the variable $s=\sigma+it$ to be in the range 
\begin{equation}
\label{range}
\sigma\geqs 1,\,\,\,\,\,\qquad1/\log Y\leqs t \leqs \log Y.
\end{equation}
We now approximate the product over primes by the reciprocal of the zeta function. The first few details of this follow Lemma 3.4 of \cite{HT} but we shall present them anyway for clarity.  

Let
\[\zeta(s,Y)=\prod_{p\leqs Y}\bigg(1-\frac{1}{p^s}\bigg)^{-1}.\]
Then, 
\[-\frac{\zeta^{\prime}(s,Y)}{\zeta(s,Y)}=\sum_{p\leqs Y}\sum_{m\geqs 1}\frac{\log p}{p^{ms}}=\sum_{n\leq Y}\frac{\Lambda(n)}{n^s}+\sum_{\substack{p\leqs Y,m>1\\p^m>Y}}\frac{\log p}{p^{ms}}.\]
On splitting this last sum over $p\leqs \sqrt{Y}$ and $\sqrt{Y}<p\leqs Y$ we see that
\[-\frac{\zeta^{\prime}(s,Y)}{\zeta(s,Y)}=\sum_{n\leqs Y}\frac{\Lambda(n)}{n^s}+O(Y^{1/2-\sigma}).\]
 Applying the usual arguments of contour integration along with Vinogradov's zero free region we find that
\[\sum_{n\leqs Y}\frac{\Lambda(n)}{n^s}=-\frac{\zeta^\prime(s)}{\zeta(s)}+\frac{Y^{1-s}}{1-s}+O\big((\log Y)^{-A}\big)\] 
for any positive, fixed $A$.  

Therefore, upon integrating over the horizontal line from $\sigma=1$ to $\sigma=\log Y$ we find that
\begin{equation}\label{prod asymp}
\begin{split}
\log\zeta(&1+it,Y)\\
=&\log\zeta(1+it)+\int_{1}^{{\log Y}}\frac{Y^{1-\sigma-it}}{1-\sigma-it}d\sigma-\log\frac{\zeta(\log Y+it)}{\zeta(\log Y+it,Y)}+O\big((\log Y)^{-(A-1)}\big)\\
=&\log\zeta(1+it)+\int_{1}^{{\log Y}}\frac{Y^{1-\sigma-it}}{1-\sigma-it}d\sigma+O\big((\log Y)^{-(A-1)}\big)
\end{split}
\end{equation}

Now, for the integral in \eqref{I_2} over the range $t\in [1,\log Y]$ we use the bound
\begin{equation}
\label{exp int bound}
\int_{1}^{{\log Y}}\frac{Y^{1-\sigma-it}}{1-\sigma-it}d\sigma\ll \frac{1}{|t|\log Y}\ll \frac{1}{\log Y}
\end{equation}
so that 
\[\log\zeta(1+it,Y)=\log\zeta(1+it)+O(1/\log Y).\]
Hence, upon exponentiating 
\[\bigg|\prod_{p\leqs Y}\bigg(1-\frac{1}{p^{1+it}}\bigg)\bigg|\asymp\frac{1}{|\zeta(1+it)|}.\]
Applying this along with the estimate
\[\frac{1}{|\zeta(1+it)|}\ll \log(2+|t|),\,\,\,\,|t|\geqs 1\]
 gives
\begin{equation*}
\begin{split}
\int_{1}^{\log Y} A(it,-it,Y)\bigg|\prod_{p\leqs Y}\bigg(1-\frac{1}{p^{1+it}}\bigg)\bigg|\frac{dt}{t^2} \ll\int_{1}^{\log Y} \log (2+|t|)\frac{dt}{t^2}
\ll1.
\end{split}
\end{equation*}

Over the range $t\in [1/\log Y,1]$ we use the asymptotic formulae
\[A(it,-it,Y)=A(0,0,Y)+O(t)=a_1+O(t),\]
where $a_1$ is given by \eqref{c_beta}, and
\[\exp\bigg(\Re\int_1^{\log Y}\frac{Y^{1-\sigma-it}}{1-\sigma-it}d\sigma\bigg)=1+O\bigg(\int_1^{\log Y}\frac{Y^{1-\sigma}}{|1-\sigma-it|}d\sigma\bigg),\]
which is valid since by \eqref{exp int bound} the integral in the exponential is $\ll 1$ for $t\geqs 1/\log Y$. We also use the formula
\[\frac{1}{|\zeta(1+it)|}=t+O(t^2).\]
Putting these together via \eqref{prod asymp} gives
\begin{equation*}
\begin{split}
\frac{1}{\pi}&\int_{1/\log Y}^1 A(it,-it,Y)\bigg|\prod_{p\leqs Y}\bigg(1-\frac{1}{p^{1+it}}\bigg)\bigg|\frac{dt}{t^2}\\
=&\frac{1}{\pi}\int_{1/\log Y}^1 A(it,-it,Y)\frac{1}{|\zeta(1+it)|}\exp\bigg(\Re\int_1^{\log Y}\frac{Y^{1-\sigma-it}}{1-\sigma-it}d\sigma\bigg)\\
&\qquad\times\Big(1+O((\log Y)^{-A-1})\Big)\frac{dt}{t^2}\\
=&\frac{a_1}{\pi}\int_{1/\log Y}^1\frac{dt}{t}+O\bigg(\int_{1/\log Y}^1\int_1^{\log Y}\frac{Y^{1-\sigma}}{|1-\sigma-it|}d\sigma\frac{dt}{t}\bigg)+O(1)\\
=&\frac{a_1}{\pi}\log\log Y+O\bigg(\log Y\int_1^{\log Y}{Y^{1-\sigma}}d\sigma\bigg)+O(1)\\
=&\frac{a_1}{\pi}\log\log Y+O(1).
\end{split}
\end{equation*}

\end{proof}

\section{Acknowledgments}

The author would like to thank Andrew Granville and Kristian Seip for their useful comments and suggestions on a draft of this paper.  He would also like to thank Oleksiy Klurman for several discussions on this topic.

%On the one hand, we can express $I_3$ as a sum over smooth numbers as follows. Since 
%\[A(s_1,s_2,Y)=\prod_{p\leqs Y}\frac{(1-p^{-1})^{1/4}(1-p^{-1-s_1-s_2})}{(1-p^{-1-s_1})^{1/2}(1-p^{-1-s_2})^{1/2}}\sum_{\substack{e_1+f_1=\\e_2+f_2}}\frac{d_{-1/2}(p^{e_1})d_{-1/2}(p^{e_2})}{p^{\frac{1}{2}e_1+\frac{1}{2}e_2+(\frac{1}{2}+s_1)f_1+(\frac{1}{2}+s_2)f_2}}\]
%we have
%\begin{equation*}
%\begin{split}
%I_3=&\frac{1}{2\pi i}\int_{c_1-iT}^{c_1+iT}A(s,-s,Y)\prod_{p\leqs Y}\bigg(1-\frac{1}{p^{1+s}}\bigg)^{1/2}\bigg(1-\frac{1}{p^{1-s}}\bigg)^{1/2}\frac{ds}{s^2}\\
%=&\bigg[\prod_{p\leqs Y}\bigg(1-\frac{1}{p}\bigg)^{5/4}\bigg]\frac{1}{2\pi i}\int_{c_1-i\infty}^{c_1+i\infty}\bigg[\sum_{\substack{m_1n_1=m_2n_2\\m_j,n_j\in S(Y)}}\frac{d_{-1/2}(m_1)d_{-1/2}(m_2)}{(m_1m_2n_1n_2)^{1/2}}\bigg]\bigg(\frac{n_2}{n_1}\bigg)^s\frac{ds}{s^2}+\cdots\\
%=&\bigg[\prod_{p\leqs Y}\bigg(1-\frac{1}{p}\bigg)^{5/4}\bigg]\sum_{\substack{m_1n_1=m_2n_2\\m_j,n_j\in S(Y)\\n_1\leqs n_2}}\frac{d_{-1/2}(m_1)d_{-1/2}(m_2)}{(m_1m_2n_1n_2)^{1/2}}\log(n_2/n_1)
%\end{split}
%\end{equation*}

\end{document}